\documentclass[12pt]{amsart}
\usepackage{amssymb,latexsym}
\usepackage{enumerate}
\usepackage{amsfonts}
\usepackage{amssymb}
\usepackage{amsmath}
\usepackage{graphicx}
\usepackage{graphicx,color}

\makeatletter
\@namedef{subjclassname@2010}{%
  \textup{2010} Mathematics Subject Classification}
\makeatother

\begin{document}

\title[Continuity of halo functions associated to density bases]{Continuity of halo functions associated to homothecy invariant density bases}

\author{Oleksandra Beznosova}
\address{Department of Mathematics, Baylor University, Waco, Texas 76798}
\email{oleksandra\!\hspace{.018in}\_\,beznosova@baylor.edu}

\author{Paul Hagelstein}
\address{Department of Mathematics, Baylor University, Waco, Texas 76798}
\email{paul\!\hspace{.018in}\_\,hagelstein@baylor.edu}
\thanks{This work was partially supported by a grant from the Simons Foundation (\#208831 to Paul Hagelstein).}




\begin{abstract}
Let $\mathcal{B}$ be a collection of open sets in
$\mathbb{R}^{n}$ such that, for any $x \in \mathbb{R}^{n}$, there exists a set $U \in \mathcal{B}$ of arbitrarily small diameter \mbox{containing $x$.}  $\mathcal{B}$ is said to be a \emph{density basis} provided that, given a measurable set $A \subset \mathbb{R}^{n}$, for a.e. $x \in \mathbb{R}^{n}$ we have
$$\lim_{k \rightarrow \infty}\frac{1}{|R_{k}|}\int_{R_{k}}\chi_{A} = \chi_{A}(x)$$
holds for any sequence of sets $\{R_{k}\}$ in $\mathcal{B}$ containing $x$ whose diameters tend to 0. The geometric maximal operator
$M_{\mathcal{B}}$ associated to $\mathcal{B}$ is defined on $L^{1}(\mathbb{R}^n)$ by
$M_{\mathcal{B}}f(x) = \sup_{x \in R \in
\mathcal{B}}\frac{1}{|R|}\int_{R}|f|$.  The  \emph{halo function} $\phi$ of $\mathcal{B}$ is defined on $(1,\infty)$ by
$$\phi(u) = \sup \left\{\frac{1}{|A|}\left|\left\{x \in \mathbb{R}^{n} : M_{\mathcal{B}}\chi_{A}(x) > \frac{1}{u}\right\}\right| : 0 < |A| < \infty\right\}\;$$
and on $[0,1]$ by $\phi(u) = u$.
    It is shown that the halo function associated to any homothecy invariant density basis is a continuous function on $(1,\infty)$.  However, an example of a homothecy invariant density basis is provided such that the associated halo function is not continuous at 1.

\end{abstract}

\maketitle

\newcommand{\ind}{1\hspace{-2.3mm}{1}}

\section{Introduction}
\newtheorem{thm}{Theorem}

Let $\mathcal{B}$ be a collection of sets of positive measure in
$\mathbb{R}^{n}$.  Moreover, suppose that for each point $x$ in
$\mathbb{R}^{n}$ there exist members of $\mathcal{B}$ of
arbitrarily small diameter containing $x$.  A natural question to
pose is: for which functions $f$ do we have

\begin{equation}\label{1}
\liminf_{\substack{diam R \rightarrow 0 \\x \in R \in
\mathcal{B}}} \frac{1}{|R|}\int_{R}f = \limsup_{\substack{diam R
\rightarrow 0
\\x \in R \in \mathcal{B}}} \frac{1}{|R|}\int_{R}f = f(x) \;\;\;\;\textup{a. e.
?}
\end{equation}
Of course, the answer to this question largely depends on
$\mathcal{B}$ itself.  If $\mathcal{B}$ were, for instance, the
collection of cubes in $\mathbb{R}^{n}$, (\ref{1}) would hold for
any $f$ in $L^{1}(\mathbb{R}^{n})$.  If $\mathcal{B}$ were
the larger class consisting of all rectangular parallelepipeds in
$\mathbb{R}^{n}$ whose sides were parallel to the axes, (\ref{1})
would hold for all $f$ in $L(\log L)^{n-1}(\mathbb{R}^{n})$ but not hold for
some functions in $L(\log L)^{n-2}(\mathbb{R}^{n})$.  If $\mathcal{B}$ were
the collection of all rectangular parallelepipeds in
$\mathbb{R}^{n}$, (\ref{1}) would fail even for some functions in
$L^{\infty}(\mathbb{R}^{n})$.   (Proofs of these results may be found in \cite{G2}.)

If (\ref{1}) holds for every function lying in the class $L_{\Phi}$ of functions $f$ such that $\int_{\mathbb{R}^{n}}\Phi(|f|) < \infty$, then  $\mathcal{B}$ is said to \emph{differentiate the
integral of any} $f$ \emph{in} $L_{\Phi}(\mathbb{R}^{n})$, or, more colloquially,
differentiate $L_{\Phi}(\mathbb{R}^{n})$.  Whether or not $\mathcal{B}$
differentiates $L_{\Phi}(\mathbb{R}^{n})$ is closely linked to the behavior of the
associated geometric maximal operator $M_{\mathcal{B}}$, defined
by
$$
M_{\mathcal{B}} f(x) = \sup_{x \in R \in \mathcal{B}}
\frac{1}{|R|}\int_{R}|f|\;.
$$
It is rather easily shown that if $\Phi$ is a Young's function and $M_{\mathcal{B}}$ is of weak type $(\Phi, \Phi)$, i.e.
$$\left|\left\{x : M_{\mathcal{B}}f(x) > \alpha\right\}\right| \leq C \int \Phi\left(\frac{|f|}{\alpha}\right)\;,$$
 then
$\mathcal{B}$ differentiates $L_{\Phi} (\mathbb{R}^{n})$.  Moreover, as was
shown in \cite{Gu1}, if $\mathcal{B}$ is homothecy invariant the
converse also holds.  


A  boundedness property of a geometric maximal quite a bit weaker than a weak type $(\Phi, \Phi)$ estimate is a so-called \emph{Tauberian estimate}.  In particular, for a given $0 < \alpha < 1$ we say that the maximal operator $M_{\mathcal{B}}$ satisfies a Tauberian estimate with respect to $\alpha$ if
$$\left|\left\{x : M_{\mathcal{B}}\chi_{A}(x) > \alpha\right\}\right| \leq C |A|$$ holds for all measurable $A \in \mathbb{R}^{n}$, where the constant $C$ is independent of $A$.  It is important to appreciate here that $C$ does depend on $\alpha$ and can generally be expected to tend to infinity as $\alpha$ tends to 0.  The optimal $C$ with respect to $\frac{1}{u}$ for $u\in (1,\infty)$ is given by the \emph{halo function associated to $\mathcal{B}$} :
$$\phi(u) = \sup \left\{\frac{1}{|A|}\left|\left\{x \in \mathbb{R}^{n} : M_{\mathcal{B}}\chi_{A}(x) > \frac{1}{u}\right\}\right| : 0 < |A| < \infty\right\}\;.$$
Following the usual convention, we extend the halo function $\phi$ to $[0,1]$ by setting $\phi(u) = u$ for $u\in [0,1]$.

The halo function $\phi$ associated to a basis $\mathcal{B}$  provides considerable information regarding the differentiation properties of $\mathcal{B}$.   Busemann and Feller showed in \cite{busemannfeller1934} that, provided $\mathcal{B}$ is homothecy invariant, the finiteness of its halo function $\phi(u)$ for all $u$ in $[0,\infty)$ holds if and only if (1) holds for all $f \in L^{\infty}(\mathbb{R}^{n})$.    (A basis $\mathcal{B}$ satisfying such a condition is said to differentiate $f \in L^{\infty}(\mathbb{R}^{n})$ and is called a \emph{density basis}.)   Bounds on the growth of the halo function $\phi(u)$ are able to yield better differentiation properties.   For example, Soria showed in \cite{Soria85} that, if $\phi(u) \leq c_{0}u(1 + \log u)^{m}$ for some non-negative constants $m$, $c_{0}$, then $\mathcal{B}$ differentiates $L(\log^{+}L)^{m}\log^{+}\log^{+}L(\mathbb{R}^{n})$.   (Further estimates along these lines may be found in the subsequent paper \cite{sjso03} of Sj\"olin and Soria.)

Motivated by the known sharp weak type bounds of the Hardy-Littlewood and strong maximal operators, mathematicians working in the area of differentiation of integrals have long suspected the following:
\\

{\bf{The Halo Conjecture:}}
Let $\mathcal{B}$ be a homothecy invariant collection of open sets in $\mathbb{R}^{n}$ and let $\phi$ be the halo function associated to $\mathcal{B}$.   Then $\mathcal{B}$ differentiates $L_{\phi}(\mathbb{R}^{n})$.
\\

Although significant work on the Halo Conjecture has been done by, among others, de Guzm\'an, Hayes, Pauc, Sj\"olin, and Soria (see for instance \cite{G2, Gu1, Hayes66, HayesPauc55,  sjso03, Soria85}), the status of the Halo Conjecture appears to be far from resolution.
\\

One difficulty regarding the Halo Conjecture is a lack of general understanding regarding structural properties of halo functions.  Partial progress on this front was made recently by P. Hagelstein and A. Stokolos, who proved in \cite{HS} quantitative bounds on the growth of a class of halo functions that enabled them to ascertain that, provided $\mathcal{B}$ is a homothecy invariant basis of convex sets, finiteness of $\phi(u)$ for \emph{any} value of $u > 1$ was enough to imply that $\mathcal{B}$ differentiates $L^{p}(\mathbb{R}^{n})$ for sufficiently large $p$.  (This turns out to not only be of intrinsic interest but also have implications regarding $L^{p}$ bounds of certain multiplier operators in harmonic analysis, see in this regard \cite{Ba, CF, HS2}.)  However, many rather na\"ive questions regarding properties of halo functions remain unanswered.   The purpose of this paper is to address one of these; namely, the issue of continuity of halo functions.  In particular, we will show that, provided $\mathcal{B}$ is a homothecy invariant density basis, the halo function $\phi$ associated to $\mathcal{B}$ must be continuous on $(1,\infty)$.   However, we shall also see that the halo function associated to a homothecy invariant density basis need not be continuous at 1.  We will also indicate an application of the proof of the main result to issues involving semicontinuity of halo functions associated to bases of rectangles, as well as  suggest directions for further research in this area.

\section{Continuity of Halo Functions}

\begin{thm}
Let $\mathcal{B}$ be a homothecy invariant density basis consisting of open sets in $\mathbb{R}^{n}$, and let $\phi$ be the halo function associated to $\mathcal{B}$ defined on $(1,\infty)$ by
$$\phi(u) = \sup \left\{\frac{1}{|A|}\left|\left\{x \in \mathbb{R}^{n} : M_{\mathcal{B}}\chi_{A}(x) > \frac{1}{u}\right\}\right| : 0 < |A| < \infty\right\}\;.$$
Then $\phi$ is a continuous function on $(1,\infty)$.
\end{thm}
\begin{proof}
We first observe that, since $\mathcal{B}$ is a homothecy invariant density basis, we are guaranteed that $\phi(u)$ is finite for every $u > 1$.  (A proof of this may be found in Chapter III of \cite{G2}.)

Let now $0 < \alpha < 1$.  Since $\phi(\frac{1}{\alpha})$ is a nonincreasing function of $\alpha$, it suffices to prove the following lemma, as together they rule out the possibility of a jump discontinuity in $\phi(\frac{1}{\alpha})$ as a function of $\alpha$.

\newtheorem{lem}{Lemma}
\begin{lem}
 Let $\mathcal{B}$ be a homothecy invariant density basis consisting of open sets in $\mathbb{R}^{n}$, and let $0 < \alpha < 1$.
\\

\noindent i)   Suppose for some finite constant $C$ we have
\begin{equation}\label{e2}
\left|\left\{x : M_{\mathcal{B}}\chi_{E}(x) > \alpha\right\}\right| \leq C|E|
\end{equation}
holds for all measurable sets $E$ in $\mathbb{R}^{n}$.  Then, for the same constant $C$, we have
\begin{equation}\label{e3}
\left|\left\{x : M_{\mathcal{B}}\chi_{E}(x) \geq \alpha\right\}\right| \leq C|E|
\end{equation}
holds for all measurable sets $E$ in $\mathbb{R}^{n}$.
\\

\noindent ii)   Suppose for some finite constant $C$ we have
\begin{equation}\label{e4}
\left|\left\{x : M_{\mathcal{B}}\chi_{E}(x) \geq \alpha\right\}\right| \leq C|E|
\end{equation}
holds for all measurable sets $E$ in $\mathbb{R}^{n}$.  Then for any $\epsilon > 0$, there exists $\delta > 0$ such that, for the same constant $C$, we have
\begin{equation}\label{e5}
\left|\left\{x : M_{\mathcal{B}}\chi_{E}(x) > \alpha - \delta \right\}\right| \leq (C + \epsilon)|E|
\end{equation}
holds for all measurable sets $E$ in $\mathbb{R}^{n}$.
\\
\end{lem}

\noindent {\it{Proof of Lemma 1.}}
\\

We first prove part (i) of the lemma.  Suppose (3) did not hold for all measurable sets $E$.  Then for some measurable set $E$ (which we now fix) we must have
$$
\left|\left\{x : M_{\mathcal{B}}\chi_{E}(x) \geq \alpha\right\}\right| > C|E|\;.$$

Let $\tilde{\epsilon} > 0$ be such that
$$
 \left|\left\{x : M_{\mathcal{B}}\chi_{E}(x) \geq \alpha\right\}\right| > (C + \tilde{\epsilon})|E|\;.$$
 Note for every $\epsilon > 0$ we have
 $$
 \left|\left\{x : M_{\mathcal{B}}\chi_{E}(x) > \alpha - \epsilon\right\}\right| > (C + \tilde{\epsilon})|E|\;.$$

Let now $0 < \epsilon < \min(\frac{\alpha}{2}, 1 - \alpha)$.  By Fatou's lemma there exists \mbox{$\left\{R_{j}\right\}_{j=1}^{N} \subset \mathcal{B}$} so that
 $$
 \frac{1}{|R_{j}|}\int_{R_{j}}\chi_{E} > \alpha - \epsilon
 $$
 holds for each $j$ and such that
 $$
 \left|\cup_{j=1}^{N}R_{j}\right| > \left(C + \frac{\tilde{\epsilon}}{2}\right)|E|\;.
 $$
   Since $\cup_{j=1}^{N}R_{j}$ is a finite union of open sets, there exists a measurable set $E' \subset \cup_{j=1}^{N}R_{j} - E$ satisfying
   $$
   \frac{1}{|R_{j} - E|}\int_{R_{j} - E}\chi_{E'} = \frac{1}{1 - \alpha}\epsilon
   $$
for each $j$. Let $\tilde{E} = E \cup E'$.  Setting $c = \frac{1}{1 - \alpha}$ we have that for each $R \in \{R_{j}\}$,
\begin{align}
\frac{1}{|R|}\int_{R}\chi_{\tilde{E}} &= \frac{1}{|R|}\left[ |E \cap R| + c\epsilon(|R| - |E\cap R|)\right] \notag
\\&= \frac{1}{|R|}\left[c\epsilon|R| + |E \cap R|(1 - c\epsilon)\right] \notag
\\&> \frac{1}{|R|}\left[c\epsilon|R| + (\alpha - \epsilon)|R|(1 - c\epsilon)\right] \notag
\\&\geq \alpha + \epsilon\left[c(1 - \alpha) - 1\right] \notag
\\&\geq \alpha . \notag
\end{align}

So $\frac{1}{|R|}\int_{R}\chi_{\tilde{E}} > \alpha$.  Note now that since $\phi(\frac{2}{\alpha}) < \infty$, there exists a finite constant $C_{\alpha / 2}$ such that $|\left\{x : M_{\mathcal{B}}\chi_{A}(x) > \frac{\alpha}{2}\right\}| \leq C_{\alpha / 2}|A|$ holds for all measurable $A$, and accordingly
\begin{align}
|\tilde{E}| &\leq |E| + c\epsilon\left|\cup_{j=1}^
{N}R_{j}\right| \notag
\\&\leq |E| + c\epsilon C_{\alpha / 2}|E|\notag
\end{align}
as $\cup R_{j} \subset \left\{x : M_{\mathcal{B}}\chi_{E}(x) > \alpha - \epsilon\right\} \subset \left\{x : M_{\mathcal{B}}\chi_{E}(x) > \frac{\alpha}{2}\right\}$.  Moreover
\begin{align}
\left|\left\{x : M_{\mathcal{B}}\chi_{\tilde{E}}(x) > \alpha \right\}\right| &\geq \left|\cup_{j=1}^{N}R_{j}\right| \notag
\\&> \left(C + \frac{\tilde{\epsilon}}{2}\right)|E| \notag
\\&\geq \left(C + \frac{\tilde{\epsilon}}{2}\right)\frac{|\tilde{E}|}{1 + c \epsilon C_{\alpha / 2}}\notag
\\&> \left(C + \frac{\tilde{\epsilon}}{4}\right) \left|\tilde{E}\right| \notag
\end{align}
for $\epsilon$ sufficiently small.  But then
$$
\left|\left\{x : M_{\mathcal{B}} \chi_{\tilde{E}} (x) > \alpha\right\}\right| \geq \left(C + \frac{\tilde{\epsilon}}{4}\right) \left|\tilde{E}\right|\;,
$$
contradicting (2).

The proof of (ii) follows along similar lines.  We proceed by contradiction.  Suppose (5) did not hold for all measurable sets $E$.  Then there would exist an $\epsilon > 0$ (which we now fix) such that, for any $\delta > 0$, there exists $E_{\delta}$ such that
$$
\left|\left\{x : M_{\mathcal{B}}\chi_{E_{\delta}}(x) > \alpha - \delta\right\}\right| > (C + \epsilon)|E_{\delta}|\;.$$

Let now $0 < \delta < \min(\frac{\alpha}{2}, 1 - \alpha)$ and $E_{\delta}$ the set associated to $\delta$ as above.  Let $\left\{R_{j}\right\}_{j=1}^{N} \subset \mathcal{B}$ be such that
$$
\frac{1}{|R_{j}|} \int_{R_{j}}\chi_{E_{\delta}} > \alpha - \delta
$$
and
$$
\left|\cup_{j=1}^{N}R_{j}\right| > \left(C + \frac{\epsilon}{2}\right)\left|E_{\delta}\right|\;.
$$
Let now $E_{\delta}' \subset \cup_{j=1}^{N}R_{j} - E_{\delta}$ be a measurable set satisfying
$$
\frac{1}{\left|R_{j} - E_{\delta}\right|}\int_{R_{j} - E_{\delta}} \chi_{E_{\delta}'} = \frac{1}{1-\alpha} \delta\;
$$
for each $j$.

Let  $c = \frac{1}{1 - \alpha}$, and let $\tilde{E_{\delta}} = E_{\delta} \cup E'_{\delta}$.  Observe that if $R \in \left\{R_{j}\right\}$ we have
\begin{align}
\frac{1}{|R|}\int_{R}\chi_{\tilde{E}_{\delta}} &\geq \frac{1}{|R|}\left[|E_{\delta} \cap R| + c\delta(|R| - |E_{\delta} \cap R|)\right] \notag
\\&= \frac{1}{|R|}[c\delta|R| + |E_{\delta}\cap R|(1 - c\delta)] \notag
\\&\geq \frac{1}{|R|}[c\delta|R| + (\alpha - \delta)|R|(1 - c\delta)] \notag
\\&\geq \alpha + [-\delta  + c\delta(1-\alpha)]\notag
\\&\geq \alpha\;.\notag
\end{align}
So $\frac{1}{|R|}\int_{R}\chi_{\tilde{E}_{\delta}} \geq \alpha$.  Note also
\begin{align}
\left|\tilde{E}_{\delta}\right| \leq |E_{\delta}| + c \delta |\cup_{j=1}^{N}R_{j}| \notag
\\ \leq |E_{\delta}| + c \delta C_{\alpha / 2}|E_{\delta}| \notag
\end{align}
as $\cup R_{j} \subset \left\{x : M_{\mathcal{B}}\chi_{E_{\delta}}(x) > \frac{\alpha}{2}\right\}$,
and
\begin{align}
\left|\left\{x : M_{\mathcal{B}}\chi_{\tilde{E}_{\delta}}(x) \geq \alpha\right\}\right| &\geq |\cup_{j=1}^{N}R_{j}| \notag
\\&> \left(C + \frac{\epsilon}{2}\right)|E_{\delta}|\notag
\\&> \left(C + \frac{\epsilon}{2}\right) \frac{|\tilde{E}_{\delta}|}{1 + c \delta C_{\alpha / 10}} \notag
\\&> \left(C + \frac{\epsilon}{4}\right)|\tilde{E}_{\delta}| \notag
\end{align}
 for $\delta > 0$ sufficiently small, contradicting (4) and completing the proof of the lemma and the theorem.
\end{proof}
\section{Further Remarks}

\noindent {\bf{1.}}\;  The above theorem regarding the continuity of halo functions associated to density bases finds the following nice application regarding the semicontinuity of a Tauberian condition associated to a homothecy invariant basis of rectangular parallelepipeds satisfying a Tauberian condition at a particular constant.

\begin{thm}
Let $\mathcal{B}$ be a homothecy invariant collection of rectangular parallelepipeds in $\mathbb{R}^{n}$.  Suppose for some $0 < \gamma < 1$ the maximal operator $M_{\mathcal{B}}$ satisfies the Tauberian condition
$|\left\{x : M_{\mathcal{B}}\chi_{E}(x) > \gamma\right\}| \leq C|E|$ for all measurable sets \mbox{$E \subset \mathbb{R}^{n}$.}  Then $M_{\mathcal{B}}$ moreover satisfies the inequality
$$|\left\{x : M_{\mathcal{B}}\chi_{E}(x) \geq \gamma\right\}| \leq C|E|$$ for all measurable sets $E$ in $\mathbb{R}^{n}$.
\end{thm}

\begin{proof}
Let $E$ be a measurable set in $\mathbb{R}^{n}$.  We inductively
 define $\mathcal{H}^{k}_{\mathcal{B}, \gamma}(E)$ for $k = 0,1,2,\ldots$ by setting
  $\mathcal{H}^{0}_{\mathcal{B},\gamma}(E) = E$
  and
\[\mathcal{H}^{k}_{\mathcal{B},\gamma}(E) = \left\{x :
 M_{\mathcal{B}}\chi_{\mathcal{H}^{k-1}_{\mathcal{B},\gamma}(E)}(x) \geq \gamma \right\}
 \]
for $k \geq 1$.

Define $\tilde{\gamma}$ by $\tilde{\gamma} = \gamma + \frac{1}{2}(1 - \gamma)$.  Note $0 < \gamma < \tilde{\gamma} < 1$ and that
$\mathcal{H}_{\mathcal{B}, \tilde{\gamma}}^{1}(A) \subset \left\{x : M_{\mathcal{B}}\chi_{A}(x) > \gamma\right\}$ holds for all measurable $A \subset \mathbb{R}^{n}$.  By the Tauberian condition on $M_{\mathcal{B}}$ we also then have that
$$
\left|\mathcal{H}_{\mathcal{B}, \tilde{\gamma}}^{1}(A)\right| \leq C|A|
$$
holds for all measurable $A \subset \mathbb{R}^{n}$.

Now, by a lemma of Hagelstein and Stokolos in \cite{HS}, we have that
if $R \in \mathcal{B}$ and \mbox{$\frac{1}{|R|}\int_{R}\chi_E =
\alpha < \gamma$,} then $R \subset \mathcal{H}^{K_{\alpha,
\gamma}}_{\mathcal{B},\gamma}(E)$ for some constant $K_{\alpha,
\gamma}$ depending only on $n$, $\alpha$, and $\gamma$, with in particular

\begin{equation}
K_{\alpha, \gamma} =
\left\lceil\frac{-\log(\frac{\gamma}{\alpha})}{\log
\gamma}\right\rceil \cdot \left\lceil 2 + \frac{ \log^{+}(\gamma
\cdot 2^n)}{\log(\frac{1}{\gamma})}\right\rceil + 1.
\end{equation}
This implies  that $\mathcal{B}$ forms a density basis.   To see this, let $0 < \alpha < \gamma$ and let $E$ be a measurable set in $\mathbb{R}^{n}$.  Let $R$ be a member of $\mathcal{B}$ such that $\frac{1}{|R|}\int_{R}\chi_{E} > \alpha$.
Then $R \subset \mathcal{H}^{K_{\alpha,\tilde{\gamma}}}_{\mathcal{B},\tilde{\gamma}}(E)$ and in particular
\begin{align}
\left|\left\{x : M_{\mathcal{B}}\chi_{E}(x) > \alpha\right\}\right| &\leq \left|\mathcal{H}^{K_{\alpha,\tilde{\gamma}}}_{\mathcal{B},\tilde{\gamma}}(E)\right|\notag
\\&\leq C \left|\mathcal{H}^{K_{\alpha,\tilde{\gamma}} - 1}_{\mathcal{B},\tilde{\gamma}}(E)\right| \notag
\\&\leq ... \leq C^{K_{\alpha,\tilde{\gamma}}}|E|\;. \notag
\end{align}
Accordingly we have $\left|\left\{x : M_{\mathcal{B}} \chi_{E}(x) > \alpha\right\}\right| \leq C^{K_{\alpha,\tilde{\gamma}}}|E|$, and hence that $\mathcal{B}$ is a density basis.  By the lemma above, the desired result follows.
\end{proof}

\noindent{\bf{2.}}\;    We remark that the statement of the above theorem was used in the proof of Proposition 1 of \cite{HS} without explicit justification; we are pleased to have provided it here.  Hagelstein and Stokolos thank Teresa Luque for bringing this issue to their attention.

A closely related and open problem is the following:
 \\

\noindent {\it{Problem:}}\; Suppose $\mathcal{B}$ is a collection of open sets in $\mathbb{R}^{n}$ (not necessarily forming a density basis) and the associated maximal operator $M_{\mathcal{B}}$ satisfies the Tauberian condition $|\left\{x : M_{\mathcal{B}}\chi_{E}(x) > \gamma\right\}| \leq C|E|$ for all measurable sets \mbox{$E \subset \mathbb{R}^{n}$.}  Must $M_{\mathcal{B}}$ satisfy the inequality
$$|\left\{x : M_{\mathcal{B}}\chi_{E}(x) \geq \gamma\right\}| \leq C|E|$$ for all measurable sets $E$ in $\mathbb{R}^{n}$?
\\

\noindent {\bf{3.}}\;  Although the halo function of a density basis is defined on $[0,\infty)$ and has been seen to be continuous on $(1,\infty)$, it is not necessarily continuous at 1, as is seen by the following example:
\\

\noindent{\it{Example:}}  Let $\mathcal{B}$ consist of all homothecies of sets in $\mathbb{R}$ of the form
$$((0,1) \cup (x, x+ \epsilon))\cap(0,2)$$ where $0 < \epsilon < 1$.  $M_{\mathcal{B}}$ is  dominated by twice the Hardy-Littlewood maximal operator and hence is bounded on $L^{p}(\mathbb{R})$ for $1 < p \leq \infty$.   Thus $\mathcal{B}$ forms a density basis. Observe however that $M_{\mathcal{B}}\chi_{(0,1)} = 1$ on $(0,2)$ and hence $\lim_{u \rightarrow 1^{+}} \phi(u) \geq 2$, so that $\phi(u)$ is discontinuous at $1$.

Of course, the collection $\mathcal{B}$ does not consist solely of convex sets, suggesting the following problem:
\\

\noindent {\it{Problem:}}\;  Let $\mathcal{B}$ be a homothecy invariant density basis of convex sets in $\mathbb{R}^{n}$, and let $\phi(u)$ be the halo function of the associated maximal operator $M_{\mathcal{B}}$.  Must $\phi(u)$ be continuous at 1?
\\

A. A. Solyanik proved in \cite{solyanik} that the halo functions of both the Hardy-Littlewood and strong maximal operators are indeed continuous at 1.  We wish to thank A. Stokolos for informing us of this result.

\end{document}